\documentclass[a4paper, 12pt]{article}
\usepackage[utf8]{inputenc}

\usepackage[T2A]{fontenc}    
\usepackage[left=2cm,right=2cm,
    top=2cm,bottom=2cm,bindingoffset=0cm]{geometry}

\usepackage[utf8]{inputenc}
\usepackage{amsfonts, amscd, amsmath, amsthm}

\usepackage{amscd,amsmath, amssymb, fancyhdr, epsfig}

\usepackage{amsmath,amsfonts,amssymb,amsthm} % AMS
\usepackage{amsrefs}
\usepackage{colortbl}
\usepackage{color}

\usepackage[vcentermath]{youngtab}
\usepackage{multirow, multicol}
\usepackage{enumerate}
\usepackage{wrapfig}
\usepackage[all,cmtip]{xy}
\usepackage[T1]{fontenc}
\usepackage{bbold}
\usepackage{tikz}
\tikzstyle{V}=[draw, fill =black, circle, inner sep=0pt, minimum size=1.5pt]
\tikzstyle{bV}=[draw, fill =black, circle, inner sep=0pt, minimum size=4pt]
\tikzstyle{cV}=[draw, fill =white, circle, inner sep=0pt, minimum size=4pt]
\tikzstyle{over}=[draw=white,double=black,line width=2pt, double distance=.75pt]
\tikzstyle{diagram}=[line width=.75pt, scale=\SCALE]

\usepackage{euscript}	 % Шрифт Евклид
\usepackage{mathrsfs} % Красивый матшрифт

\numberwithin{equation}{section}

\def\eqref#1{(\ref{#1})}
\newcommand{\goth}{\mathfrak}
\newcommand{\g}{{\mathfrak g}}
\newcommand{\gl}{{\mathfrak {gl}}}

\def\1{\sqrt{-1}\:}

\newcommand{\cntrct}                % contraction with a vector field
{\hspace{2pt}\raisebox{1pt}{\text{$\lrcorner$}}\hspace{2pt}}

\renewcommand{\tilde}{\widetilde}

\renewcommand{\phi}{\varphi}
\renewcommand{\epsilon}{\varepsilon}

\newcommand{\End}{\operatorname{End}}
\newcommand{\Rep}{\operatorname{Rep}}

\newcommand{\Res}{\operatorname{Res}}
\newcommand{\Hom}{\operatorname{Hom}}

\newcounter{Mycounter}[section]
\newcounter{lemma}[section]
\newcounter{claim}[section]
\newcounter{sublemma}[section]
\newcounter{corollary}[section]
\newcounter{theorem}[section]
\newcounter{conjecture}[section]
\newcounter{definition}[section]
\newcounter{example}[section]
\newcounter{remark}[section]
\newcounter{problem}[section]
\newcounter{question}[section]
\def\blacksquare{\hbox{\vrule width 5pt height 5pt depth 0pt}}

\begin{document}
%%%%%%%%%%%%%%%%%%%%%%%%%%%%%%%%%%%%%%%%%%%%%%%%%%%%%%%%%%%%

\begin{center}
    \textbf{\large{Olshanski's Centralizer Construction and Deligne Tensor Categories}}
\end{center}

\begin{center}
    Alexandra Utiralova
\end{center}

\begin{abstract}
The family of Deligne tensor categories $\Rep(GL_t)$ is obtained from the categories $\mathbf{Rep}~GL(n)$ of finite dimensional representations of groups $GL(n)$ by interpolating the integer parameter $n$ to complex values. Therefore, it is a valuable tool for generalizing classical statements of representation theory. In this work we introduce and prove the generalization of Olshanski's centralizer construction of the Yangian  $Y(\gl_n)$. Namely, we prove that for generic $t\in\mathbb{C}$ the centralizer subalgebra of $GL_t$-invariants in the universal enveloping algebra $U(\mathfrak{gl_{t+n}})$ is the tensor product of $Y(\mathfrak{gl}_n)$ and the center of $U(\mathfrak{gl_{t}})$. The main feature of this construction is that it does not involve passing to a limit, contrary to the original construction of Olshanski.
\end{abstract}

\section{Introduction.}

~~~~Categories of finite dimensional representations of classical groups such as  $GL(n), O(n), Sp(n), \\S_n$, form  families of categories algebraically depending on a parameter $n$. The first definition for interpolation of these categories to complex values of $n$ for the group $GL(n)$ -- the family of categories $\Rep(GL_t)$ -- was given in the classical text by Deligne and Milne \textit{Tannakian categories} \cite{DM}. In his further works \cite{D2}, \cite{D3} Deligne developed this construction: he showed, for instance, that these categories are semisimple for all non-integer $t$ and proved the universal property which will be discussed below. He also defined interpolation categories for other classical groups: $\Rep(S_t)$, $\Rep(O_t), \Rep(Sp_t)$  \cite{D1}, that are known now together with $\Rep(GL_t)$ as Deligne  categories. Like in any symmetric tensor category, there are well defined notions of associative algebras, Lie algebras, Hopf algebras, etc. in Deligne categories. Such algebras generalize the notion of algebras with an action of a group, some examples of which for the group $GL(n)$ are the universal enveloping algebra $U(\gl_n)$ or the Yangian $Y(\gl_n)$.

We are interested in the categories $\Rep(GL_t)$. For these we will introduce the analogues of the Yangian and the universal enveloping algebra and will state and prove a generalization of the centralizer construction of the Yangian originally proved by Grigori Olshanski in \cite{O1}. Namely, the centralizer construction describes the tensor product $Y(\mathfrak{gl}_n)\otimes A_0$ (where $A_0=\mathbb{C}[x_1,x_2,\ldots]$ is the free polynomial algebra with the generators of degrees $1,2,3,\ldots$) as the limit of the centralizer subalgebras $U(\mathfrak{gl}_{N+n})^{GL_N}$ as $N\to\infty$. A justification and detailed description of the aforementioned construction could be found, for example, in the book by A. Molev \textit{Yangians and classical Lie algebras} \cite{M}. We prove the following

\noindent{\bf Main Theorem.}  \emph{For any transcendental $t\in\mathbb{C}$ we have $Res_{GL_{t}}^{GL_{t+n}}U(\mathfrak{gl}_{t+n})^{GL_t}\simeq Y(\mathfrak{gl}_n)\otimes A_0$.} 

Similar statement was proved by Molev and Olshanski for twisted Yangians in \cite{MO}. We believe that this result can also be generalized for Deligne categories $\Rep(O_t), \Rep(Sp_t)$.

\subsection*{The paper is organized as follows.} In Section~2 we recall the definitions and main properties of Deligne categories and Yangians and define Yangian of $\mathfrak{gl}_t$. In Section~3 we give a precise formulation of the Main Theorem. In Section~4 we prove the Main Theorem. The Appendix (Section~5) contains technical results on the abelian envelope of the Deligne category $\Rep(GL_t)$ used in the present paper.

\subsection*{Acknowledgements.}

I want to thank Leonid Rybnikov for suggesting this problem to me and for all the discussions we had about it and Pavel Etingof for all the valuable advice.

\section{\large{Some definitions and preliminaries.}} 

~~~~It is well known that the group $GL(n)$ has the fundamental irreducible  $n-$dimensional representation $V$ that is faithful, and therefore, all other irreducible finite dimensional representations can be realized as sub- and factor-modules of representations of the form $V^{\otimes k}\otimes V^{*\otimes l}$. Now, knowing the fundamental theorem of invariant theory and Schur-Weyl duality, we can easily describe morphisms between the representations of this kind. It motivates the following definition:

\begin{definition} Define the rigid additive symmetric tensor category $\widetilde{\Rep}(GL_t)$, generated by a single object of dimension $t$. That is, the category with objects of the form $[k,l]:=V^{\otimes k}\otimes V^{*\otimes l}$ and all possible finite direct sums of such objects. The generating object $V$ would be an analogue of the fundamental representation of the group $GL(n)$, therefore, similarly to the classical case we require: 
$$
\Hom(V^{\otimes k}\otimes V^{*\otimes l}, V^{\otimes p}\otimes V^{*\otimes q}) = 0,
$$if $k+q \neq l+p$. In the case when $k+q=l+p=n$, $$\Hom(V^{\otimes k}\otimes V^{*\otimes l}, V^{\otimes p}\otimes V^{*\otimes q})\simeq \Hom(V^{\otimes n}, V^{\otimes n})$$ 
is isomorphic to $\mathbb CS_n$ with the standard action of $S_n$ on $V^{\otimes n}$.

By "of dimension $t$" we mean that the composition of coevaluation and evaluation
$$
\begin{CD}
\mathbb{1}@>\mathrm{coev}>>V\otimes V^* @>\mathrm{ev}>>\mathbb{1}
\end{CD}
$$
is the multiplication by $t$.

Morphisms between objects in $\widetilde{\Rep}(GL_t)$ are linear combinations of diagrams of the following kind:

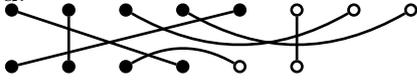
\begin{tikzpicture}[line width=1pt, scale=.75]
	\foreach\x in {1,...,5}{
		\node[bV] (t\x) at (\x,1){};}
	\foreach\x in {1,...,4}{
		\node[bV] (b\x) at (\x,0){};}
	\foreach\x in {6,...,8}{
	    \node[cV] (t\x) at (\x,1){};}
	\foreach\x in {5,...,6}{
	    \node[cV] (b\x) at (\x,0){};}
	\draw [bend left] (t8) to (t4) (b3) to (b5) (t7) to (t3);
	\draw (b1)--(t5) (b2)--(t2) (b4)--(t1) (b6)--(t6);
	%\foreach\x in {1,...,4}{\node[below] at (b\x) {V};}
	%\foreach\x in {5,...,6}{\node[below] at (b\x) {V^*};}
	%\foreach\x in {1...5}{\node[above] at (t\x) {V};}
	%\foreach\x in {6,...,8}{\node[above] at (t\x) {V^*};}
	%\foreach\x in {1,...,8}{
		%\node[above] at (\x+.5,1) {\small{$\otimes$}};}
	%\foreach\x in {1,...,6}{
	%	\node[below] at (\x+.5,0) {\small{$\otimes$}};}
	
    %\foreach\x in {1,...,5}{\node[bV] at (t\x){};\node[bV] at (b\x){};}
\end{tikzpicture}

Here, $\bullet$ stands for $V$ and $\circ$ stands for $V^*$; vertical straight lines going from top to bottom are the identity morphisms; curved lines at the top row $\bullet\smile\circ$ stand for the evaluation, at the bottom row - $\bullet\frown \circ$ - coevaluation. In our case it represents the map $[5,3]\to [4,2]$.

To compose two such morphisms one should place the second diagram below (Pic.1) and compose the edges, then delete every cycle, keeping in mind that it is a multiplication by $t$ (Pic.2):  
$$
$$

\begin{tikzpicture}[line width=1pt, scale=.75]
	\foreach\x in {1,...,5}{
		\node[bV] (t\x) at (\x,1){};}
	\foreach\x in {1,...,4}{
		\node[bV] (b\x) at (\x,0){};}
	\foreach\x in {6,...,8}{
	    \node[cV] (t\x) at (\x,1){};}
	\foreach\x in {5,...,6}{
	    \node[cV] (b\x) at (\x,0){};}
	\draw [bend left] (t8) to (t4) (b3) to (b5) (t7) to (t3);
	\draw (b1)--(t5) (b2)--(t2) (b4)--(t1) (b6)--(t6);
\end{tikzpicture}

\begin{tikzpicture}[line width=1pt, scale=.75]
	\foreach\x in {1,...,4}{
		\node[bV] (t\x) at (\x,1){};}
	\foreach\x in {1,...,3}{
		\node[bV] (b\x) at (\x,-1){};}
	\foreach\x in {4}{
	    \node[cV] (b\x) at (\x,-1){};}
	\foreach\x in {5,...,6}{
	    \node[cV] (t\x) at (\x,1){};}

	\draw [bend left] (t5) to (t3);
	\draw (b2)--(t1) (b4)--(t6) (b1)--(t4) (b3)--(t2);
	\node at (4,-2) {Pic.1};
\end{tikzpicture}

\begin{tikzpicture}[line width=1pt, scale=.75]
	\foreach\x in {1,...,5}{
		\node[bV] (t\x) at (\x,1){};}
	\foreach\x in {6,...,8}{
	    \node[cV] (t\x) at (\x,1){};}
	\foreach\x in {1,...,3}{
		\node[bV] (b\x) at (\x,-2){};}
	\foreach\x in {4}{
	    \node[cV] (b\x) at (\x,-2){};}
	\draw [bend left] (t8) to (t4) (t7) to (t3);
	\draw (b3)--(t2) (b2)--(t5) (b1)--(t1) (b4)--(t6);
    \node[left] at (.5,-.7) {t};
    \node at (4,-3) {Pic.2};
\end{tikzpicture}

\textit{\underline{Note:}} The algebra $\End([k,l])$ is known as the Walled Brauer Algebra $B_{k,l}(t)$. 
$$
$$

\textit{The Deligne category $\Rep(GL_t)$} is the Karoubi envelope of the category $\widetilde{\Rep}(GL_t)$.
\end{definition}
\vspace{.1cm}

This category is semisimple (and hence abelian) for $t\notin \mathbb Z$ and has the following universal property: \textit{if $\mathcal C$ is a rigid symmetric additive tensor category then isomorphism classes of additive symmetric monoidal functors  $\Rep(GL_t)\to\mathcal C$ are in bijection with isomorphism classes of objects of dimension $t$ in $\mathcal C$. It means that any such functor is defined up to an isomorphism by the image of $V$.}\cite{E2}

Using this property we can define the functors
\begin{enumerate}
    \item[$\bullet$] $\Res_n: \Rep(GL_{t+n}) \to \Rep(GL_t), \Res_n(V_{(t+n)})=\mathbb{1}^n\oplus V_{(t)};$
    
    \item[$\bullet$] $ \Res_{s,t}: \Rep(GL_{t+s}) \to \Rep(GL_t)\boxtimes \Rep(GL_s),$ $ \Res_{s,t}(V_{(t+s)})= V_{(s)}\boxtimes \mathbb{1}\bigoplus \mathbb{1}\boxtimes V_{(t)};$
    
    \item[$\bullet$] $Ev_n: \Rep(GL_{t=n}) \to \mathbf{Rep}~GL(n), Ev_n(V) = \mathbb C^n.$
\end{enumerate}

\vspace{.1cm}
\begin{definition} Let $\Rep(GL_t)^{ab}$ be the abelian envelope of $\Rep(GL_t)$. It was shown  in \cite{H}, that it exists and is also a symmetric tensor category.
\end{definition}

\vspace{.1cm}

\begin{lemma} Let $\mathcal C$ be an abelian symmetric tensor category over a field of characteristic zero, $V\in Ob~\mathcal C$ and $A$ a filtered algebra in $\mathrm{Ind}(\mathcal C) - $ the Ind-completion of the category $\mathcal C$  generated by $V$ in degree 1 such that $gr~A=SV$. Let $\mathcal D$ be the Karoubian subcategory in $\mathcal C$ generated by tensor powers of $V$. Then for each $i$ the object $F^iA$ is in $\mathcal D$.
\end{lemma}

\begin{proof}

We have a surjective map $TV\to A$, let $I$ be its kernel. And let $I_0=gr~I$. Then $TV=I_0\oplus SV$, hence, $TV=I\oplus SV$. Therefore, $A=TV/I$ is isomorphic to $SV$ as a filtered Ind-object. 
\end{proof}

\vspace{.1cm}

\begin{definition} The algebra $\gl_t = V^*\otimes V$ - is an algebra in the category $\Rep(GL_t)$ with a multiplication $m$:

\xymatrix{
&(V^*\otimes V)\otimes (V^*\otimes V) \ar[rr]^-{m = \mathrm{ev}_{23}} &&V^*\otimes V 
}

\begin{tikzpicture}[line width=1pt, scale=.75]
	\foreach\x in {2,4}{
		\node[bV] (t\x) at (\x,1){};}
	\foreach\x in {1,3}{
		\node[cV] (t\x) at (\x,1){};}
	\foreach\x in {3}{
	    \node[bV] (b\x) at (\x,-1){};}
	\foreach\x in {2}{
	    \node[cV] (b\x) at (\x,-1){};}

	\draw [bend left] (t3) to (t2);
	\draw (b3)--(t4) (b2)--(t1);
	\node at (.3,.2) {$m:$};
\end{tikzpicture}

Now we can define the Lie algebra $\goth{gl}_t$ (or just $\g)$ as the Lie algebra \\$V^*\otimes V$ in the category $\Rep(GL_t)$ with a commutator $c$:

\xymatrix{
&(V^*\otimes V)\otimes (V^*\otimes V) \ar@{.>}[rd] \ar[rr]^-{c = m - m\circ P} &&V^*\otimes V \\
&&\Lambda^2(V^*\otimes V)\ar@{.>}[ru]} 

\noindent where $P$ is the permutation of the factors of the form $(V^*\otimes V)$.

\vspace{.1cm}

For any object $W\in \mathrm{Ind}(\Rep(GL_t))$ the tensor algebra $TW$ is a well defined Ind-object in our category isomorphic to
$$
\mathbb{1}\oplus W \oplus (W\otimes W) \oplus (W\otimes W\otimes W) \oplus...
$$
with standard multiplication.

Define the universal enveloping algebra $U(\g)$  to be a quotient of the tensor algebra:
$$
U(\g) := T(\g)/(f(\g\otimes\g)),
$$
where $f = i_2 - i_2P - i_1c$, where $i_k:\g^{\otimes k}\to T(\g)$ is just the inclusion of the k-th graded component.

A priori $U(\g)$ is a filtered algebra in $\mathrm{Ind}(\Rep(GL_t)^{ab})$ (the filtration on $U(\g)$ is inherited from the grading on the tensor algebra). But  Poincaré–Birkhoff–Witt (PBW) theorem holds for Lie algebras in any symmetric tensor category over a field of characteristics zero, therefore, the conditions of Lemma 2.3. are satisfied for $A=U(\g)$. Hence, $U(\g)$ is indeed a well-defined Ind-object in $\Rep(GL_t)$ and $\mathrm{gr}~U(\g)$ is isomorphic to the symmetric algebra $S(\g)$. 

It follows that the invariants in both algebras coincide: $\Hom(\mathbb{1}, U(\g))=\Hom(\mathbb{1},S(\g))$  and they are isomorphic to $A_0$, where $A_0 = \mathbb C[x_1,x_2,\ldots]$, where each $x_k$ has degree $k$ (which was proved in \cite{E2}).

\end{definition}

 Now we will define the Yangian $Y(\gl_t)$ as an algebra in the category $\mathrm{Ind}(\Rep(GL_t))$ as follows.  Let $Y^{gen} = \bigoplus_{i\in\mathbb{Z}_{>0}}(V^*\otimes V)_i$ be the space of generators of this algebra and let $\widetilde{Y} = T(Y^{gen})$.

Define the map $a_0:\mathbb{1}\to V\otimes V^*$ and the maps $a_i:\mathbb{1}\to (V^*\otimes V)_i\otimes (V\otimes V^*)$ as follows: $a_0=\mathrm{coev}$ and $a_i=\mathrm{coev}_{13}\mathrm{coev}_{24}$.

\begin{tikzpicture}[line width=1pt, scale=.75]
	\foreach\x in {2,3}{
		\node[bV] (t\x) at (\x,1){};}
	\foreach\x in {1,4}{
	    \node[cV] (t\x) at (\x,1){};}
	\draw [bend left] (t1) to (t3) (t2) to (t4);
    \node[left] at (0,1) {$a_i:$};
    \node at (.7,1){$($};
    \node at (2.4,1){$)_i$};
\end{tikzpicture}

Now we can define $\tilde T(u)\in \Hom(\mathbb{1},\tilde Y\otimes (V\otimes V^*))[[u^{-1}]]$  as $\tilde T(u)=\sum a_iu^{-i}$. Note that since $\tilde Y$ and $V\otimes V^*$ are algebras in our category (clearly, $V\otimes V^*$ is isomorphic to $\gl_t^{op}$), the vector space $\tilde B:=\Hom(\mathbb{1},\tilde Y\otimes(V\otimes V^*))$ is an ordinary associative algebra.

Remembering the classical RTT-relation for the Yangian, define the $R-$ matrix in \\$\Hom(\mathbb{1},(V\otimes V^*)\otimes(V\otimes V^*))[[u^{-1}]]$ as $R(u)=\mathrm{coev}_{12}\mathrm{coev}_{34}-u^{-1}\mathrm{coev}_{14}\mathrm{coev}_{23}=1-u^{-1}P$. 

\begin{tikzpicture}[line width=1pt, scale=.75]
	\foreach\x in {1,3,7,9}{
		\node[bV] (t\x) at (\x,1){};}
	\foreach\x in {2,4,8,10}{
		\node[cV] (t\x) at (\x,1){};}
    \node at (5,1) {$-$};
    \node at (6,1) {$u^{-1}$};
	\draw [bend left] (t1) to (t2) (t3) to (t4) (t7) to (t10) (t8) to (t9);
	\node at (-.5,1) {$R(u)=$};
\end{tikzpicture}

Let $\tilde A$ be an algebra $\Hom(\mathbb{1},\tilde Y\otimes (V\otimes V^*)\otimes (V\otimes V^*))[[u^{-1},v^{-1}]]$. We can give the embedding of the elements we have just defined into $\tilde A$:

\xymatrix{
\widetilde{T}(u)_1:\mathbb{1}\ar[rrrrr]^-{\widetilde{T}(u)_{12}\otimes \mathrm{coev}_3}&&&&& \underset{1}{\tilde Y}\otimes\underset{2}{(V\otimes V^*)}\otimes\underset{3}{(V\otimes V^*)},\\
\widetilde{T}(v)_2:\mathbb{1}\ar[rrrrr]^-{\widetilde{T}(v)_{13}\otimes \mathrm{coev}_2}&&&&& \underset{1}{\tilde Y}\otimes\underset{2}{(V\otimes V^*)}\otimes\underset{3}{(V\otimes V^*)},\\
{R}(u-v):\mathbb{1}\ar[rrrrr]^-{{R}(u-v)_{23}\otimes (\mathbf{1}_{\tilde Y})_1}&&&&& \underset{1}{\tilde Y}\otimes\underset{2}{(V\otimes V^*)}\otimes\underset{3}{(V\otimes V^*)},
}

(with two different embeddings of $\tilde T(u)$).

I.e. we have defined $(a_i)_1$ and $(a_i)_2$ for each $i$:

\begin{tikzpicture}[line width=1pt, scale=.75]
	\foreach\x in {2,3,5}{
		\node[bV] (t\x) at (\x,1){};}
	\foreach\x in {1,4,6}{
	    \node[cV] (t\x) at (\x,1){};}
	\draw [bend left] (t1) to (t3) (t2) to (t4) (t5) to (t6);
    \node[left] at (0,1) {$(a_i)_1:$};
    \node at (.7,1){$($};
    \node at (2.4,1){$)_i$};
    \foreach\x in {2,3,5}{
		\node[bV] (b\x) at (\x,-1){};}
	\foreach\x in {1,4,6}{
	    \node[cV] (b\x) at (\x,-1){};}
	\draw [bend left] (b1) to (b5) (b2) to (b6) (b3) to (b4);
    \node[left] at (0,-1) {$(a_i)_2:$};
    \node at (.7,-1){$($};
    \node at (2.4,-1){$)_i$};
\end{tikzpicture}

Now we are ready to define the Yangian $Y(\gl_t)$, which we will refer to as $Y_t$ or just $Y$ when there is no ambiguity.

\begin{definition} Consider the evaluation $\epsilon=\mathrm{ev}_{12}:\tilde Y\otimes(\underset{1}{V\otimes V^*})^{\otimes 2}\otimes \underset{2}{(V^*\otimes V})^{\otimes 2}\to \tilde Y$. \\And consider a map $f= R(u-v)\tilde T(u)_1\tilde T(v)_2-\tilde T(v)_2\tilde T(u)_1 R(u-v): \mathbb{1}\to \tilde Y\otimes(V\otimes V^*)^{\otimes 2}$. Let $I\subset \tilde Y$ be an ideal generated by the image of the following composition:

\xymatrix{
(V^*\otimes V)^{\otimes 2}\ar[r]^-{f\otimes id} & \tilde Y\otimes (V\otimes V^*)^{\otimes 2}\otimes(V^*\otimes V)^{\otimes 2}\ar[r]^-\epsilon&\tilde Y 
}

By Theorem 5.2 (see the Appendix), the conditions of Lemma 2.3 hold for $\tilde Y/I$ inside $\mathrm{Ind}(\Rep(GL_t)^{ab})$. Therefore, $Y=\tilde Y/I$ is a well-defined Ind-object of $\Rep(GL_t)$. 

\end{definition}
\vspace{.1cm}

Let $T(u)$ be the composition of $\tilde T(u)$ with the projection $\tilde Y\otimes(V\otimes V^*)\to Y\otimes(V\otimes V^*)$. It is an element of the algebra $ B[[u^{-1}]]:=\Hom(\mathbb{1},Y\otimes (V\otimes V^*))[[u^{-1}]]$. Its zero term $a_0=\mathbf{1}_Y\otimes \mathrm{coev}_{23}$ is the unit of the algebra $B:=\Hom(\mathbb{1},Y\otimes(V\otimes V^*))$. Therefore, $T(u)$ is invertible.

Let $\tau: \Hom(\mathbb{1},X\otimes (V\otimes V^*))\to \Hom(V^*\otimes V,X)$ for any object $X$ be the natural isomorphism. Note that $\tau(a_i):V^*\otimes V\to (V^*\otimes V)_i$ is the identity map for $i>0$ and $\tau(a_0):V^*\otimes V\to\mathbb{1}$ is the evaluation. 

Let $\tilde T'(u)\in \tilde B[[u^{-1}]]$ be any invertible element, $\tilde T'(u)=\sum b_iu^{-i}$. Then $\tilde T'(u)$ defines an automorphism (which we denote $\tilde T(u)\mapsto \tilde T'(u)$) of $\tilde Y$, which is defined on generators as follows: $\tau(b_i):(V^*\otimes V)_i=(V^*\otimes V)\to \tilde Y$ (basically, each $b_i$ determines where the $i'$th generator goes). It is easy to see that $\tilde T(u)$ defines the identity morphism on $\tilde Y$.

Now if we want an invertible element $T'(u)\in B[[u]]$ to define an automorphism of $Y$, we should require that $R(u-v)T'(u)_1T'(v)_2=T'(v)_2T'(u)_1R(u-v)$ inside \\ $A:=\Hom(\mathbb{1},Y\otimes(V\otimes V^*)^{\otimes 2})$. 

Similarly, $T'(u)$ defines an anti-automorphism of $Y$ if $R(u-v)T'_2(v)T'_1(u)=T'_1(u)T'_2(v)R(u-v)$. 

\vspace{.2cm}

\begin{lemma} The shift $T(u)\mapsto T(u+s)$ is an automorphism of $Y$ for any $s\in \mathbb C$. The maps $T(u)\mapsto T(-u)$ and $T(u)\mapsto T(u)^{-1}$ are anti-automorphisms.
\end{lemma}

\begin{proof}

\begin{enumerate}
    \item[1.]The equation $$R(u-v)T(u+s)_1T(v+s)_2=T(v+s)_2T(u_s)_1R(u-v)$$ is obtained by the change of variables $u+s\mapsto u, v+s\mapsto v$.

 \item[2.] For $T(u)\mapsto T(-u)$ we note that $P:\mathbb{1}\to (V\otimes V^*)\otimes (V\otimes V^*), P=\mathrm{coev}_{14}\mathrm{coev}_{23}$ commutes with $R(u-v)$, since $R(u-v)=\mathbf{1}-(u-v)^{-1}P$. Clearly, $P^2=\mathbf{1}$. It is also easy to see that $PT(u)_1P=T(u)_2$. Therefore, $$R(u-v)T(-v)_2T(-u)_1 = PR(u-v)PT(-v)_2P^2T(-u)_1P^2 = PR(-v-(-u))T(-v)_1T(-u)_2P =$$ $$= PT(-u)_2T(-v)_1R(-v-(-u))P = PT(-u)_2P^2T(-v)_1PR(u-v) = T(-u)_1T(-v)_2R(u-v).$$

 \item[3.] $T(u)\mapsto T(u)^{-1}$ is an anti-automorphism since $$R(u-v)T(v)_2^{-1}T(u)_1^{-1} = T(u)_1^{-1}T(v)_2^{-1}R(u-v)$$ is clearly equivalent to the RTT relation $R(u-v)T(u)_1T(v)_2 = T(v)_2T(u)_1R(u-v)$.\

 (Side note: the map $T(u)\mapsto T(u)^{-1}$ is actually the antipode of $Y$.)

 \end{enumerate}
\end{proof}

\vspace{.15cm}

\begin{corollary} \textit{Let $\omega_t$ be a map defined by $T(-u-t)^{-1}$. Then $\omega_t$ is an automorphism of $Y$}.
\end{corollary}

\vspace{.19cm}

We can rewrite the relation for $U(\g)$ in the following way: instead of the map\\ $f=i_2-i_2P-i_1c:\g\otimes\g\to T(\g)$

\begin{tikzpicture}[line width=1pt, scale=.5]
	\foreach\x in {2,4,8,10,14,16,20,22}{
		\node[bV] (t\x) at (\x,1){};}
	\foreach\x in {1,3,7,9,13,15,19,21}{
		\node[cV] (t\x) at (\x,1){};}
	\foreach\x in {2,4,8,10,15,21}{
	    \node[bV] (b\x) at (\x,-1){};}
	\foreach\x in {1,3,7,9,14,20}{
	    \node[cV] (b\x) at (\x,-1){};}
    \node at (0,.3) {$f=$};
    \node at (5.2,.3){$-$};
    \node at (11.2,.3){$-$};
    \node at (17.2,.3){$+$};
	\draw [bend left] (t15) to (t14) (t22) to (t19);
	\foreach\x in {1,2,3,4}{
	    \draw (b\x)--(t\x);}
	\draw (b7)--(t9) (b8)--(t10) (b9)--(t7) (b10)--(t8) (b14)--(t13) (b15)--(t16) (b20)--(t21) (b21)--(t20);
\end{tikzpicture}

\noindent
we can consider a map $f':\mathbb{1}\to T(\g)\otimes \g^*)^{\otimes 2}=T(V^*\otimes V)\otimes (V\otimes V^*)^{\otimes 2}$

\begin{tikzpicture}[line width=1pt, scale=.5]
	\foreach\x in {2,4,5,7,12,14,15,17}{
		\node[bV] (t\x) at (\x,1){};}
	\foreach\x in {1,3,6,8,11,13,16,18}{
	    \node[cV] (t\x) at (\x,1){};}
	\draw [bend left] (t1) to (t5) (t2) to (t6) (t3) to (t7)  (t4) to (t8) (t11) to (t17) (t12) to (t18) (t13) to (t15) (t14) to (t16);
    \node[left] at (0,1) {$f'=$};
    \node at (.7,1){$($};
    \node at (4.4,1){$)$};
    \node at (9.2,1){$-$};
    \node at (19.2,1){$-$};
    \node at (10.7,1){$($};
    \node at (14.4,1){$)$};  
    \foreach\x in {2,3,5,10,11,13}{
		\node[bV] (b\x) at (\x,-1){};}
	\foreach\x in {1,4,6,9,12,14}{
	    \node[cV] (b\x) at (\x,-1){};}
	\draw [bend left] (b1) to (b3) (b2) to (b6) (b4) to (b5) (b9) to (b13) (b10) to (b12) (b11) to (b14);
	\node at (0.7,-1){$($};
    \node at (2.4,-1){$)$};  
    \node at (8.7,-1){$($};
    \node at (10.4,-1){$)$};  
    \node at (0,-1){$-$};
    \node at (7.2,-1){$+$};
    \node at (14.5,-1.2) {.};

\end{tikzpicture}

 If we identify $\g$ with $(V^*\otimes V)_1$, then this map is equal to $$(a_1)_1(a_1)_2-(a_1)_2(a_1)_1-P(a_1)_2+(a_1)_2P,~~~~(\star)$$ because 

\begin{tikzpicture}[line width=1pt, scale=.75]
	\foreach\x in {2,3,5}{
		\node[bV] (t\x) at (\x,1){};}
	\foreach\x in {1,4,6}{
	    \node[cV] (t\x) at (\x,1){};}
	\draw [bend left] (t1) to (t3) (t2) to (t6) (t4) to (t5);
    \node[left] at (0,1) {$P(a_1)_2=$};
    \node at (.7,1){$($};
    \node at (2.4,1){$)_1$};
    \foreach\x in {2,3,5}{
		\node[bV] (b\x) at (\x,-1){};}
	\foreach\x in {1,4,6}{
	    \node[cV] (b\x) at (\x,-1){};}
	\draw [bend left] (b1) to (b5) (b2) to (b4) (b3) to (b6);
    \node[left] at (0,-1) {$(a_1)_2P=$};
    \node at (.7,-1){$($};
    \node at (2.4,-1){$)_1$};
\end{tikzpicture}

Now consider  the coefficient of $u^{-1}v^{-1}$ in the RTT-relation $$R(u-v)T(u)_1T(v)_2  -T(v)_2T(u)_1R(u-v) = 0.$$

Clearly, it is the same as the coefficient of $u^{-1}v^{-1}$ in the following expression:
$$
(1-u^{-1}P)(1+u^{-1}(a_1)_1)(1+v^{-1}(a_1)_2)-(1+v^{-1}(a_1)_2)(1+u^{-1}(a_1)_1)(1-u^{-1}P).
$$

And this coefficient is equal to $(\star)$. 
 
Therefore, the restriction of the RTT-relation to the subalgebra $T((V\otimes V^*)_1)$ of $\tilde Y$ coincides with the relation for $U(\g)$. Hence, we have the embedding $U(\g)\hookrightarrow Y$ and, most importantly, the evaluation map:$$
E: Y\to U(\g),
$$that sends $(V^*\otimes V)_1$ to $\g$ and sends $(V^*\otimes V)_i$ to zero when $i>1$. Clearly, the composition of $E$ with the embedding is the identity on $U(\g)$.

\section{The main theorem.}

~~~~~To obtain Olshanski's classical centralizer construction of the Yangian $Y(\goth{gl}_n)$, we consider the centralizers $A_n(N)~:=~U(\goth{gl}_N)^{\goth{gl}_{N-n}}$, where the subalgebra $\goth{gl}_{N-n}$ is embedded canonically into $\goth{gl}_N$, i.e. it is generated by the elements $E_{i j},~~ i,j\in\{n+1,...,N\}$ (in this notation $A_0(N)$ is the center of $U(\goth{gl}_N)$). We also need the surjective restrictions $A_n(N)\to A_n(N-1)$ that respect filtrations inherited from $U(\gl_N)$ and $U(\gl_{N-1})$ correspondingly. The explicit construction for these maps could be found in Section $8.1$ of Molev's book \cite{M}. Consider the diagram

\xymatrix{
&\ldots\ar@{->>}[r]&A_n(N+1)\ar@{->>}[r]&A_n(N)\ar@{->>}[r]&...\ar@{->>}[r]&A_n(n),
}
\noindent
and let $A_n$ be its projective limit in the category of filtered algebras. Then Olshanski's theorem says that
$$
A_n \simeq Y_n\otimes A_0.
$$

The map from the Yangian to the limit is given by the family of maps $\psi_N: Y_n\to A_n(N)$ that commute with the maps in the diagram. By construction, each $\psi_N$ is the restriction of the composition of the automorphism $\omega_N: Y_N\to Y_N, T(u)\mapsto T(-u-N)^{-1}$ and the evaluation homomorphism $Y_N\to U(\gl_N)$ to the canonically embedded subalgebra $Y_n$. \cite{M}

The algebra $A_0(N)$ is the polynomial algebra in $N$ variables of degrees $1,...,N$, that correspond to the traces of degrees of the matrix $(e_{i j})$. Thus the projective limit $A_0$ is the polynomial algebra in variables  $\{x_i~|~deg~x_i = i,~i\in \mathbb Z_{\ge 0}\}$. We will keep this notation for our statement. For each $N$ we have a map $\mathcal Z_N: A_0\to A_0(N+n)\hookrightarrow U(\gl_{N+n})^{\gl_n}$ whose image is just the center of $U(\gl_{N+n})$. Then $Y_n\otimes A_0$ is mapped to $A_n(N)$ via the map $\phi_N:=\psi_N\otimes \mathcal Z_N$. The induced map to the limit $\phi_\infty: Y_n\otimes A_0 \to A_n$ is the isomorphism we wished for. 

Now consider the restriction of the universal enveloping algebra of $\goth{gl}_{t+n}$ via the restriction functor defined above:
$$
\Res_n(U(\goth{gl}_{t+n})) \in \Rep(GL_t).
$$
The analogue for the centralizer of the subalgebra  $\goth{gl}_t$ inside it would be "the invariants of the action of $GL_t$", i.e.
$$
\Res_n(U(\goth{gl}_{t+n}))^{GL_t} \stackrel{\mathrm{def}}{=} \Hom_{\Rep(GL_t)}(\mathbb{1}, \Res_n(U(\goth{gl}_{t+n}))).
$$

We will prove the following statement, which is our main result.

\begin{theorem} For any transcendental $t\in\mathbb{C}$ 
$$
\Res_n(U(\goth{gl}_{t+n}))^{GL_t} \simeq Y_n\otimes A_0.
$$

\end{theorem}

Note that our construction doesn't require passing to limits.

\begin{remark} In fact, the theorem that we are going to prove says that we have such an isomorphism for Weyl generic $t$, i.e. for all $t$ except a countable number of values. However, it follows immediately from the discussion below that if the statement of Theorem 3.1 holds for at least one transcendental $t$, then it holds for all $t\in\mathbb C\setminus \overline{\mathbb Q}$.    

The category $\Rep(GL_t)$ is linear over $\mathbb C$. Given some map of fields $\mathbb C\to K$ we can perform the base change on $\Rep(GL_t)$ turning it into $K-$linear category.

One can construct $\Rep(GL_t)$ for transcendental $t$ as a subcategory of the ultraproduct $\prod_{\mathcal U} \mathcal C_i$ of categories $\mathcal C_i=\mathbf{Rep}~ GL(n,\overline{\mathbb Q})$, which becomes a $\mathbb C$-linear category after choosing the identification $\prod_{\mathcal U}\overline{\mathbb Q}\to \mathbb C$ with $(1,2,3,4,...)\mapsto t$, where $\mathcal U$ is some nontrivial ultrafilter on $\mathbb N$. This construction was discussed in \cite{D1} and is similar  to the construction we introduce in the Appendix. 

Now for any two transcendental complex numbers $s,t$ there exists an automorphism of $\mathbb C$ over $\overline{\mathbb Q}$ that sends $t$ to $s$. The base change of $\Rep(GL_{t+N})$ via this automorphism is then (using the construction above) $\Rep(GL_{s+N})$ for any $N\in\mathbb Z$. Therefore,  
$$
\Res_n(U(\goth{gl}_{t+n}))^{GL_t} \simeq Y_n\otimes A_0.
$$
for any transcendental $t\in\mathbb C$.

\end{remark}

\section{Proof of Theorem 3.1.}

~~~~~First we want to define a map $\psi: Y_n \to \Res_n(U(\goth{gl}_{t+n}))^{GL_t}$. Let $F$ be the composition

\xymatrix{
&Y_{t+n}\ar[rr]^{\omega_{t+n}}\ar@/_1pc/@{.>}[rrrr]_{F}&&Y_{t+n}\ar[rr]^{E}&&U(\gl_{t+n}).
}

The restriction $\Res_n(Y_{t+n})$ is an algebra in the category $\Rep(GL_t)$ with the generators $((V\oplus\mathbb{1}^n)\otimes (V^*\oplus \mathbb{1}^n))_i$. It has two trivially embedded subalgebras: $Y_t$ and the "real" Yangian  $Y_n$ generated by $(\mathbb{1}^n\otimes \mathbb{1}^n)_i$. Thus, $Y_n$ is clearly a subalgebra in the algebra of invariants $\Hom(\mathbb{1}, \Res_n(Y_{t+n}))$. Whilst $U(\gl_n)$ is a subalgebra in the algebra of invariants of  $\Res_n(U(\gl_{t+n}))$. 

Then let $\psi$ be the restriction to $Y_n$ of the map $$\Hom(\mathbb{1}, \Res_n(F)): \Hom(\mathbb{1}, \Res_n(Y_{t+n}))\to \Hom(\mathbb{1}, \Res_n(U(\gl_{t+n}))).$$

Since $A_0 = \Hom(\mathbb{1}, U(\gl_{t+n}))$, we have the homomorphism $$\mathcal Z:A_0\rightarrow \Hom(\mathbb{1}, \Res_n(U(\gl_{t+n}))).$$ Therefore, we have a well defined homomorphism of algebras $\phi = \psi\otimes\mathcal Z: Y_n\otimes A_0\to \Res_n(U(\gl_{t+n}))^{GL_t}$. We want to show that it is surjective for all $t\in\mathbb C$ and injective for Weyl generic  $t$.

\subsection{Surjectivity:}

~~~~~Let $T(u) = \sum a_ku^{-k}$. It is easy to see that $$T(-u-t)^{-1}/(a_i=0, \forall i>1) = \sum_k (a_1+t)^ku^{-k}$$ ($T(-u-t)^{-1}$ gives an automorphism $\omega_{t}$ and putting all $a_i=0$ for $i>1$ means composing it with the evaluation homomorphism). 

Remember that $\tau$ is the isomorphism between $\Hom(\mathbb{1},Y\otimes(V\otimes V^*))$ and $\Hom(V^*\otimes V,Y)$. Then it is easy to compute $\tau(a_1^k)$:

\begin{tikzpicture}[line width=1pt, scale=.75]
    \foreach\x in {2}{
		\node[bV] (b\x) at (6,2){};}
    \node[cV] (b1) at (5,2){};
	\foreach\x in {2,5,10}{
		\node[bV] (t\x) at (\x,0){};}
	\foreach\x in {1,4,9}{
	    \node[cV] (t\x) at (\x,0){};}
	\foreach\x in {0.7,3.7,8.7}{
	\node at (\x,0){(};}
	\foreach\x in {2.4,5.4,10.4}{
	    \node at (\x,0){$)_1$};}
	\foreach\x in {3,6,8}{
	    \node at (\x,0){$\cdot$};}
	\node at (6.9,0){$\ldots$};
	\draw [bend left] (t2) to (t4) (t5) to (5.8,0.2) (8.2,0.2) to (t9);
	\draw (b1)--(t1) (b2)--(t10);
    \node[left] at (-1,1) {$\tau(a_1^k):$};
\end{tikzpicture}

where $\cdot$ stands for multiplication in $Y$.

The map $\psi$ acts on the generators $T^k\stackrel{\mathrm{def}}{=}((\mathbb{1}^n)^*\otimes\mathbb{1}^n)_k\subset Y(\gl_n)$ in the same way as the composition of the restriction of $\tau(\Res_n((a_1+t+n)^k))$ (the term in inner brackets  is the k-th term in $T(-u-t-n)^{-1}/(a_i=0, i>1)$) to $Y_n$ with the evaluation homomorphism to $\Res_n(U(\mathfrak{gl}_{t+n}))$ (i.e. after expanding $(a_1+(t+n))^k$ as a combination of terms of the form $a_1^i$, in the diagram above one should interpret $\bullet$ as $V\oplus\mathbb{1}^n=\Res_n(V)$ and $\circ$ as its dual,  then take the restriction to $(\mathbb{1}^n)^*\otimes \mathbb{1}^n$, and then map each term $(\circ~ ~\bullet)_1$ isomorphically to $\Res_n(\gl_{t+n})$): 

\begin{tikzpicture}[line width=1pt, scale=.75]
    \node[draw, fill =black,  inner sep=0pt, minimum size=4pt] (b2) at (6,2){};
    \node[draw, fill =white,  inner sep=0pt, minimum size=4pt] (b1) at (5,2){};
	\foreach\x in {2,5}{
		\node[bV] (t\x) at (\x,0){};}
	\foreach\x in {4,9}{
	    \node[cV] (t\x) at (\x,0){};}
	\node[draw, fill =black,  inner sep=0pt, minimum size=4pt] at (10,0){};
	\node[draw, fill =white,  inner sep=0pt, minimum size=4pt] at (1,0){};
	\foreach\x in {0.7,3.7,8.7}{
	\node at (\x,0){(};}
	\foreach\x in {2.4,5.4,10.4}{
	    \node at (\x,0){$)$};}
	\foreach\x in {3,6,8}{
	    \node at (\x,0){$\star$};}
	\node at (6.9,0){$\ldots$};
	\draw [bend left] (t2) to (t4) (t5) to (6.6,0.2) (7.4,0.2) to (t9);
	\draw (b1)--(t1) (b2)--(t10);
	\node at (5.5,-2){$+$};

	\node at (-2.5,-2) {$\psi|_{T^k}=\sum_i \binom{k}{i} (t+n)^{k-i}$};
\end{tikzpicture}

\begin{tikzpicture}[line width=1pt, scale=.75]
    \node[draw, fill =black,  inner sep=0pt, minimum size=4pt] (b2) at (6,2){};
    \node[draw, fill =white,  inner sep=0pt, minimum size=4pt] (b1) at (5,2){};
	\foreach\x in {2,5}{
		\node[draw, fill =black,  inner sep=0pt, minimum size=4pt] (t\x) at (\x,0){};}
	\foreach\x in {4,9}{
	    \node[draw, fill =white,  inner sep=0pt, minimum size=4pt] (t\x) at (\x,0){};}
	\node[draw, fill =black,  inner sep=0pt, minimum size=4pt] at (10,0){};
	\node[draw, fill =white,  inner sep=0pt, minimum size=4pt] at (1,0){};
	\foreach\x in {0.7,3.7,8.7}{
	\node at (\x,0){(};}
	\foreach\x in {2.4,5.4,10.4}{
	    \node at (\x,0){$)$};}
	\foreach\x in {3,6,8}{
	    \node at (\x,0){$\star$};}
	\node at (6.9,0){$\ldots$};
	\draw [bend left] (t2) to (t4) (t5) to (6.6,0.2) (7.4,0.2) to (t9);
	\draw (b1)--(t1) (b2)--(t10);
	\node at (-5.3,-1) {};
\end{tikzpicture}

Here, $\star$ is the multiplication in $\Res_n(U(\gl_{t+n}))$,  $\blacksquare$ and {\scriptsize $\square$}  \normalsize denote $\mathbb{1}^n$ and $(\mathbb{1}^n)^*$ correspondingly and the number of terms in brackets is equal to $i$.

We can introduce a filtration on $Y_n$ that comes from the grading on $T(Y_n^{gen}), $ where $deg~ t_{ij}^{(l)} = l$. PBW for the Yangian says that $\mathrm{gr}(Y_n)=S(Y_n^{gen})$. And the filtration on $A_0$ is induced from the grading defined before: $\mathrm{deg}~ x_k = k$. Then we have a filtration on $Y_n\otimes A_0$
$$
F^m(Y_n\otimes A_0) = \sum F^k(Y_n)\otimes F^{m-k}(A_0).
$$

To prove the surjectivity of $\phi=\psi\otimes \mathcal Z$, it is enough to show that the induced map of associated graded algebras $\psi^{gr}\otimes \mathcal Z^{gr}: S(\bigoplus_k T^k)\otimes A_0\to S((V\oplus\mathbb{1}^n)^*\otimes(V\oplus\mathbb{1}^n))^{GL_t}$ is surjective. To obtain $\psi^{gr}$ from $\psi$, we should forget about the lower order terms in $(a_1+t+n)^k$ and change the multiplication $\star$ to the commutative multiplication in $S(\Res_n(\gl_{t+n}))$, i.e. instead of the sum over all $i$ we'll get that $\psi^{gr}|_{T^k}$ is just the sum of two maps $T^k\to S^k(\Res_n(\gl_{t+n}))$: 

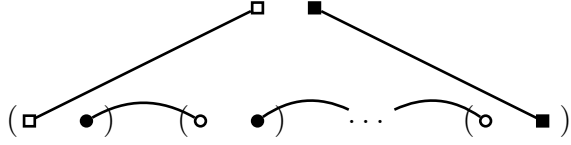
\begin{tikzpicture}[line width=1pt, scale=.75]
    \node[draw, fill =black,  inner sep=0pt, minimum size=4pt] (b2) at (6,2){};
    \node[draw, fill =white,  inner sep=0pt, minimum size=4pt] (b1) at (5,2){};
	\foreach\x in {2,5}{
		\node[bV] (t\x) at (\x,0){};}
	\foreach\x in {4,9}{
	    \node[cV] (t\x) at (\x,0){};}
	\node[draw, fill =black,  inner sep=0pt, minimum size=4pt] at (10,0){};
	\node[draw, fill =white,  inner sep=0pt, minimum size=4pt] at (1,0){};
	\foreach\x in {0.7,3.7,8.7}{
	\node at (\x,0){(};}
	\foreach\x in {2.4,5.4,10.4}{
	    \node at (\x,0){$)$};}

	\node at (6.9,0){$\ldots$};
	\draw [bend left] (t2) to (t4) (t5) to (6.6,0.2) (7.4,0.2) to (t9);
	\draw (b1)--(t1) (b2)--(t10);
	\node at (5.5,-2){$+$};

	\node at (-2,-2) {$\psi^{gr}|_{T^k}=$};
\end{tikzpicture}

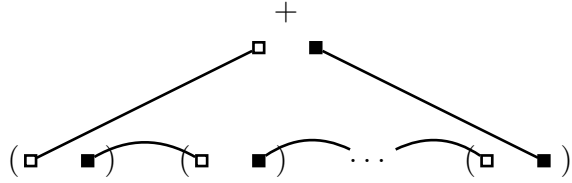
\begin{tikzpicture}[line width=1pt, scale=.75]
    \node[draw, fill =black,  inner sep=0pt, minimum size=4pt] (b2) at (6,2){};
    \node[draw, fill =white,  inner sep=0pt, minimum size=4pt] (b1) at (5,2){};
	\foreach\x in {2,5}{
		\node[draw, fill =black,  inner sep=0pt, minimum size=4pt] (t\x) at (\x,0){};}
	\foreach\x in {4,9}{
	    \node[draw, fill =white,  inner sep=0pt, minimum size=4pt] (t\x) at (\x,0){};}
	\node[draw, fill =black,  inner sep=0pt, minimum size=4pt] at (10,0){};
	\node[draw, fill =white,  inner sep=0pt, minimum size=4pt] at (1,0){};
	\foreach\x in {0.7,3.7,8.7}{
	\node at (\x,0){(};}
	\foreach\x in {2.4,5.4,10.4}{
	    \node at (\x,0){$)$};}

	\node at (6.9,0){$\ldots$};
	\draw [bend left] (t2) to (t4) (t5) to (6.6,0.2) (7.4,0.2) to (t9);
	\draw (b1)--(t1) (b2)--(t10);
	\node at (-3.05,-1) {};
\end{tikzpicture}

It is obvious that the subalgebra in $\mathrm{gr}(Y_n)$ generated by $T^1$ maps isomorphically to the subalgebra $S(\gl_n)\subset S(\Res_n(\gl_{t+n}))^{GL_t}$. This means that we a priori have all invariants of the form  
\begin{tikzpicture}[line width=1pt, scale=.75]
    
	\foreach\x in {2,5}{
		\node[draw, fill =black,  inner sep=0pt, minimum size=4pt] (t\x) at (\x,0){};}
	\foreach\x in {4,9}{
	    \node[draw, fill =white,  inner sep=0pt, minimum size=4pt] (t\x) at (\x,0){};}
	\node[draw, fill =black,  inner sep=0pt, minimum size=4pt] at (10,0){};
	\node[draw, fill =white,  inner sep=0pt, minimum size=4pt] at (1,0){};
	\foreach\x in {0.7,3.7,8.7}{
	\node at (\x,0){(};}
	\foreach\x in {2.4,5.4,10.4}{
	    \node at (\x,0){$)$};}
	\node at (6.9,0){$\ldots$};
	\draw [bend left] (t2) to (t4) (t5) to (6.6,0.2) (7.4,0.2) to (t9);

\end{tikzpicture}. So, since the second term in $\psi^{gr}_k$ maps $T^k$ into this subalgebra, we can forget about it. So, let $a_k$ be the image of the first term of $\psi^{gr}_k$:

\begin{tikzpicture}[line width=1pt, scale=.75]
	\foreach\x in {2,5}{
		\node[bV] (t\x) at (\x,0){};}
	\foreach\x in {4,9}{
	    \node[cV] (t\x) at (\x,0){};}
	\node[draw, fill =black,  inner sep=0pt, minimum size=4pt] at (10,0){};
	\node[draw, fill =white,  inner sep=0pt, minimum size=4pt] at (1,0){};
	\foreach\x in {0.7,3.7,8.7}{
	\node at (\x,0){(};}
	\foreach\x in {2.4,5.4,10.4}{
	    \node at (\x,0){$)$};}
	\node at (6.9,0){$\ldots$};
	\draw [bend left] (t2) to (t4) (t5) to (6.6,0.2) (7.4,0.2) to (t9);
	\node at (-.5,0) {$a_k=$};
\end{tikzpicture}.

The variable $x_k \in A_0$ is mapped to the sum of the following invariants in $S^k(\Res_n(\gl_{t+n})):$

\begin{tikzpicture}[line width=1pt, scale=.75]
	\foreach\x in {2,5,10}{
		\node[bV] (t\x) at (\x,0){};}
	\foreach\x in {1,4,9}{
	    \node[cV] (t\x) at (\x,0){};}
	\foreach\x in {0.7,3.7,8.7}{
	\node at (\x,0){(};}
	\foreach\x in {2.4,5.4,10.4}{
	    \node at (\x,0){$)$};}
	\node at (6.9,0){$\ldots$};
	\draw [bend left] (t2) to (t4) (t5) to (6.6,0.2) (7.4,0.2) to (t9) (t1) to (t10);
	\node at (5.5,-1) {$+$};
	\node at (-1,-1){$\mathcal Z^{gr}(x_k)=$};

\end{tikzpicture}

\begin{tikzpicture}[line width=1pt, scale=.75]

    \foreach\x in {2,5}{
		\node[draw, fill =black,  inner sep=0pt, minimum size=4pt] (t\x) at (\x,0){};}
	\foreach\x in {4,9}{
	    \node[draw, fill =white,  inner sep=0pt, minimum size=4pt] (t\x) at (\x,0){};}
    \node[draw, fill =black,  inner sep=0pt, minimum size=4pt] at (10,0){};
	\node[draw, fill =white,  inner sep=0pt, minimum size=4pt] at (1,0){};
	\foreach\x in {0.7,3.7,8.7}{
	\node at (\x,0){(};}
	\foreach\x in {2.4,5.4,10.4}{
	    \node at (\x,0){$)$};}
	\node at (6.9,0){$\ldots$};
	\draw [bend left] (t2) to (t4) (t5) to (6.6,0.2) (7.4,0.2) to (t9) (t1) to (t10);
	\node at (-2,-1){};
\end{tikzpicture}

The second term is again in $S(\gl_n)$, so we can forget it. Let $b_k$ denote the first term:

\begin{tikzpicture}[line width=1pt, scale=.75]
	\foreach\x in {2,5,10}{
		\node[bV] (t\x) at (\x,0){};}
	\foreach\x in {1,4,9}{
	    \node[cV] (t\x) at (\x,0){};}
	\foreach\x in {0.7,3.7,8.7}{
	\node at (\x,0){(};}
	\foreach\x in {2.4,5.4,10.4}{
	    \node at (\x,0){$)$};}
	\node at (6.9,0){$\ldots$};
	\draw [bend left] (t2) to (t4) (t5) to (6.6,0.2) (7.4,0.2) to (t9) (t1) to (t10);
	\node at (-1,0){$b_k=$};
	\node at (14,0){$(*)$};
\end{tikzpicture}

We want to show now the all that invariants in $S((V\oplus\mathbb{1}^n)^*\otimes(V\oplus\mathbb{1}^n))$ are in the image of $\varphi^{gr}$, i.e. they are polynomiars in $a_i,b_j$. Let us lift everything to $T(\Res_n(\gl_{t+n}))$ and check that we get all the invariants up to permutations.

Consider the $k$-th graded component $((V\oplus\mathbb{1}^n)^*\otimes(V\oplus\mathbb{1}^n))^{\otimes k}$ of the tensor algebra. Of all direct summands in this product we are interested only in those that contain equal number of terms $V$ and $V^*$,  since these are the only summands that contain invariants. For convenience let $\bullet$ denote $V$, $\circ$ -- $V^*$, \scalebox{0.6}{$\square$} -- $(\mathbb{1}^n)^*$, $\blacksquare$ --$\mathbb{1}^n$. In this notation $(V\oplus\mathbb{1}^n)^*\otimes(V\oplus\mathbb{1}^n) = \circ\bullet + \circ~\blacksquare +$ \scalebox{0.6}{$\square~$}$\bullet +$\scalebox{0.6}{$\square~~$}$\blacksquare$. We are interested in strings of symbols of length $2k$ containing equal number of black ($\bullet$) and white ($\circ$) circles. A priori white figures (\scalebox{0.6}{$\square$} and $\circ$) and black figures ($\blacksquare, \bullet$)  come in pairs, white before black.  We are interested in these strings up to permutations of such pairs (because $S_k$ acts on $((V\oplus \mathbb{1}^n)^*\otimes(V\oplus \mathbb{1}^n))^{\otimes k}$ permuting pairs and we are interested in the invariants inside the symmetric algebra $S((V\oplus \mathbb{1}^n)^*\otimes(V\oplus \mathbb{1}^n))$). Note that all the pairs of the form \scalebox{0.6}{$\square~~$}$\blacksquare$ are  factors that come from $S(\gl_n)$, so we can forget them. All the invariants come from coevaluations of $V$ and $V^*$, therefore, if our string (of length $2k$) doesn't contain any squares, then the invariants in it are monomials of the form $b_{k_1}...b_{k_l}$  of degree $k$ (modulo some terms from $S(\gl_n)$). Indeed, we start with some pair $\circ~\bullet$ and $\bullet$ should come from coevaluation, so there must be a $\circ$ to pair it with; if it is the first one, we're done, because $\circ\smile\bullet$ is just $b_1$; if not - there must be another pair of $\circ~\bullet$ and so on until we get the cycle as in (*)). Thus,  we have got a factor of the form $b_{k_1}$, where $k_1$ is the number of pairs ($\circ~\bullet)$ involved. Then we proceed by looking at any other pair $\circ~\bullet$ (that is not involved in $b_{k_1}$) and repeat the algorithm. 

If there are some squares in the string then \scalebox{0.6}{$\square$} and $\blacksquare$ come in equal numbers, so we can assume that there is a pair of the form \scalebox{0.6}{$\square$}$~\bullet$ (because we forget the pairs of the form \scalebox{0.6}{$\square~~$}$\blacksquare$). Take it to be the first pair. So $\bullet$ must come from coevaluation, so there is a  $\circ$ to pair it with. We choose the second pair  $\circ~*$ containing the aforementioned $\circ$. Now if the second factor is $\bullet$ then we again look for a $\circ$ to pair it with, and so on until we reach a square $\blacksquare$. Clearly, this gives us a factor of the form

\begin{tikzpicture}[line width=1pt, scale=.75]
	\foreach\x in {2,5}{
		\node[bV] (t\x) at (\x,0){};}
	\foreach\x in {4,9}{
	    \node[cV] (t\x) at (\x,0){};}
	\node[draw, fill =black,  inner sep=0pt, minimum size=4pt] at (10,0){};
	\node[draw, fill =white,  inner sep=0pt, minimum size=4pt] at (1,0){};
	\foreach\x in {0.7,3.7,8.7}{
	\node at (\x,0){(};}
	\foreach\x in {2.4,5.4,10.4}{
	    \node at (\x,0){$)$};}
	\foreach\x in {3,6,8}{
	    \node at (\x,0){$\otimes$};}
	\node at (6.9,0){$\ldots$};
	\draw [bend left] (t2) to (t4) (t5) to (6.6,0.2) (7.4,0.2) to (t9);
	\node at (1.5,0){$\otimes$};
	\node at (9.5,0) {$\otimes$};
\end{tikzpicture}

And this is just $a_i$. So, we have proved  the surjectivity.

\subsection{Injectivity:}

~~~~~Consider the maps from Olshanski's construction:

\xymatrix{
&Y_n\otimes A_0\ar[dl]_{\phi_\infty}\ar@{.>}[dr]^{\phi_{N+n}}\\
A_n\ar@{->>}[rr]&&A_n(N+n)
}

Since $\phi_\infty$ is an embedding and  $F^m(Y_n\otimes A_0)$ is finite dimensional, there exists $N_0$ s.t. $\phi_{N+n}$ restricted to $F^m(Y_n\otimes A_0)$ is also an embedding for any $N>N_0$.

It follows from our construction that the functor $Ev_N: \Rep(GL_t)|_{t=N}\to \mathbf{Rep}~GL(N)$ takes the Yangian $Y_t$ to $Y(\gl_N)$. Moreover, $Ev_N(Res_n(Y_{t+n}))=Y(\gl_{N+n})$ and the subalgebra $\omega_t(Y_n)$ is mapped to the subalgebra $\omega_N(Y_n)\subset Y_{N+n}$. The invariants $\Res_n(U(\gl_{t+n}))^{GL_t}$ are mapped to the invariants $U(\gl_{N+n})^{\gl_n}$. Note that the restriction of the monoidal functor $Ev_N$ to the invariants is just the map of algebras and we have the following commutative diagram: 

\xymatrix{
&Y_n\otimes A_0\ar[ld]_{\phi}\ar[rd]^{\phi_{N+n}}\\
\Res_n(U(\gl_{t+n}))^{GL_t}\ar[rr]^{Ev_N}&&U(\gl_{N+n})^{\gl_N}
}

Let $J(t)$ be the kernel of the map $\phi(t): Y_n\otimes A_0 \to \Res_n(U(\gl_{t+n}))^{GL_t}.$ Then the intersection $J(t)\cap F^m(Y_n\otimes A_0)$ is a subspace that algebraically depends on $t$. Since for big enough $N$ $\phi_{N+n}$ restricted to $F^m(Y_n\otimes A_0)$ is injective,  $\phi(N)$ restricted to this subspace is necessary injective. Therefore,  $J(N)\cap F^m(Y_n\otimes A_0) = 0$ for any $N>N_0$. Hence,  $J(t)\cap F^m(Y_n\otimes A_0) = 0$ for any $t$ in some Zariski open set in $\mathbb C$ (i.e. for all $t$, except for a finite number of points). Since $J(t) = \bigcup_m (J(t)\cap F^m(Y_n\otimes A_0))$, the kernel is zero for Weyl generic $t$.

It proves the theorem.

\section{Appendix: Ultraproduct construction and PBW theorem for the Yangian.}

We follow the work of Nate Harman \cite{H} in this section. Our goal is to prove PBW for the Yangian $Y_t$ defined in the abelian envelope of $\Rep(GL_t)$. 

Given a non-principal ultrafilter $\mathcal U$ on $\mathbb N$ and categories $\mathcal C_1,\mathcal C_2,...$ one can define the ultraproduct of these categories $\hat{\mathcal C}_{\mathcal U}:=\prod_{\mathcal U} \mathcal C_i$. The objects in this category are represented by strings of objects $X_1, X_2,...,$ s.t. $X_i\in Ob(\mathcal C_i)$ and two such strings define the same object of $\hat{\mathcal C}_{\mathcal U}$ if they agree on a set belonging to $\mathcal U$. Morphisms are represented by pointwise morphisms between such strings with a similar equivalence relation.

Let $p_1, p_2,...$ be a sequence of prime numbers tending to infinity. Fix an isomorphism $f:\prod_{\mathcal U} \overline{\mathbb F}_{p_n} \to \mathbb C$. Let $f^{-1}(t)$ be represented by $(\overline t_1, \overline t_2,...),$ with $ \overline t_n\in\overline{\mathbb F}_{p_n}$. And let $t_1, t_2,...$ be a sequence of natural numbers tending to infinity, s.t. $t_n~mod~ p_n = \overline t_n$. Let $\mathcal C_n = \mathbf{Rep}~GL(t_n, \overline{\mathbb F}_{p_n})$ and let $V_n\in Ob(\mathcal C_n)$ be the fundamental $t_n$-dimensional representation. Then we have the following statement:

\begin{theorem} Let $V$ be an object in $\hat{\mathcal C}_{\mathcal U}$ represented by the string $V_1,V_2,...$. And let $\mathcal C_{\mathcal U}$ be the subcategory generated by $V$ under the operations of taking duals, tensor products, direct sums, and direct summands and $\mathcal C^{ab}_{\mathcal U}$ - the same but also allowing taking subquotients. Then $\mathcal C_{\mathcal U}$ is equivalent to $\Rep(GL_t)$ and $\mathcal C^{ab}_{\mathcal U}$ - to its abelian envelope.

\end{theorem}
We refer the reader to the original article for more detail.

What is important for us is that the proof of PBW for the Yangian $Y_n$ (given for example in Molev's book \cite{M}) works in any characteristic. The Yangian $Y_t$ (introduced in Def.2.5.) is then well defined in $\mathrm{Ind}(\mathcal C^{ab}_{\mathcal U})$ and represented by the string $Y_{t_1}, Y_{t_2},...$, where $Y_{t_n}$ is the classical Yangian $Y(\gl_{t_n})$ over the field 
$\overline{\mathbb F}_{p_n}$. Since for all of them we have an isomorphism $\mathrm{gr}(Y_{t_n})=S(\bigoplus_k (V_n^*\otimes V_n)_k)$, the same must be true for $Y_t$ and thus we have proved the following theorem.

\begin{theorem}
Poincaré–Birkhoff–Witt theorem holds for the Yangian $Y_t$ defined in the category $\mathrm{Ind}(\mathcal C^{ab}_{\mathcal U})$, in other words, there is an isomorphism of filtered $\mathrm{Ind}$-objects $\mathrm{gr}(Y_t)=S(Y_t^{gen})$.
\end{theorem}

\end{document}